\documentclass[11pt,twoside]{article}
\usepackage{amsthm}
\usepackage{amsmath}
\usepackage{amssymb}
\usepackage{array}
\usepackage{amsfonts,amscd}
\usepackage{exscale}
\usepackage{eucal}
\usepackage{latexsym,amsmath}
\usepackage{indentfirst}
\newcommand{\e}{\varepsilon}

\numberwithin{equation}{section}
\newtheorem{Prop}{\bf Proposition}[section]

\newtheorem{defn}{\bf Definition}[section]
\newtheorem{Rem}{\bf Remark}[section]
\newtheorem{Ex}{\bf Example}[section]
\newtheorem{Th}{Theorem}[section]

\begin{document}
\def \b{\Box}

\begin{center}
{\Large {\bf ON TRANSITIVE GROUP$-$GROUPOIDS}}
\end{center}

\begin{center}
{\bf Gheorghe IVAN}
\end{center}

\setcounter{page}{1}

\pagestyle{myheadings}

{\small {\bf Abstract}. The purpose of  this paper is to study the
transitive group-groupoids. }
{}

\section{Introduction}
\indent\indent The groupoid was introduced by H. Brandt [Math.
Ann., \textbf{96}(1926), 360-366] and it is developed by P. J.
Higgins in \cite{higgi}. The notion of group-groupoid was defined
by R. Brown and Spencer in \cite{brspen}. In \cite{mivan} is given
a new definition for group-groupoids.

The groupoids and group-groupoids have important applications in
various areas of science (see for instance \cite{brown06},
\cite{icozgu}$-$\cite{ivan99},\cite{mack}$-$\cite{rare}).

 The paper is organized as follows. In Section 2  we present some basic facts
 related to group-groupoids. Section 3 is dedicated to study of transitive group-groupoids.
 A main result for to describe the commutative transitive group-groupoids is established.

\section{Basic properties of group-groupoids}
\indent\indent We start this section with a brief review of some
necessary backgrounds on groupoids \cite{higgi, brown, mack}. In
the terms of category theory,  a groupoid is a small category in
which every morphism is an isomorphism. In the following
definition we describe the concept of groupoid as algebraic
structure.
\begin{defn}
{\rm (\cite{mack})  A {\it groupoid $ G$ over} $ G_{0}$ is a pair
$( G, G_{0} )$ of nonempty sets endowed with two surjective maps
$\alpha, \beta:G \to G_{0}$, a partially binary operation $ m $
from $ G_{(2)}= \{ (x,y)\in G\times G | \beta(x) = \alpha(y) \} $
to $~G, $  an injective map $ \e: G_{0} \to G $ and a map $ i: G
\to  G $ satisfying the following properties (we write
$x\cdot y $ or $ x y$ for $ m(x,y) $ and $ x^{-1} $ for $ i(x)$):\\
$(G1)~$ ({\it associativity})$~~(xy)z=x(yz)~$ in the sense that if
one of two products $ (xy)z $ and $ x(yz) $ is defined,
then the other product is also defined and they are equals;\\
$(G2)~$ ({\it identities}) $~(\e(\alpha(x)),~(x,\e(\beta(x)))\in
G_{(2)} $ and
 $ \e(\alpha(x)) x= x= x \e(\beta(x);$\\
$(G3)~$ ({\it inverses}) $~(x^{-1},x),~(x,x^{-1}) \in G_{(2)} $
 and $ x^{-1} x = \e(\beta(x)),~x
x^{-1}=\e(\alpha(x)).$}\hfill$\Box$
\end{defn}
 A groupoid $ G $ over $ G_{0} $  with the \textit{structure functions} $ \alpha
$ ({\it source}), $ \beta $ ({\it target}), $ m $ ({\it
multiplication}), $ \varepsilon $ ({\it inclusion})  and $ i$
({\it inversion}), is denoted by $ (G, \alpha, \beta, m, \e, i, G
_0)$. Whenever we write a product in a given groupoid, we are
assuming that it is defined.

For a groupoid we sometimes write $ (G, \alpha, \beta, G_{0} ) $
or $ (G, G_{0} ) $; $ G_{0} $ is called the {\it base} of $ G $
and $ \e(G_{0})\subseteq G $ is called the {\it unit set} of $ G.
$ The element $\varepsilon(\alpha(x))$ (resp.,
$\varepsilon(\beta(x))$) is called the \textit{left}
 (resp., \textit{right unit}) of $x; x^{-1}$ is called the
\emph{inverse} of $x$.

For each $ u\in G_{0} $, the set $ \alpha^{-1}(u)$ (resp., $
\beta^{-1}(u)$) is called \textit{$\alpha-$fibre} (resp.,
\textit{$\beta-$fibre}) at $ u$. The map $ (\alpha, \beta): G \to
G_{0}\times G_{0} $ defined by $ (\alpha , \beta)(x):= (\alpha(x),
\beta(x)),~(\forall)~ x\in G $ is called the {\it anchor map} of $
G.$

A groupoid is {\it transitive}, if its anchor map is surjective.
In the case when, $ G$ is transitive, then the isotropy groups $
G(u),~(\forall)~u\in G_{0}~$  are  groups isomorphes.

For any $ u\in G_{0},$ the set $ G(u):= \alpha^{-1}(u)\cap
\beta^{-1}(u)$ is a group under the restriction of the
multiplication, called the {\it isotropy group at} $ u $ of the
groupoid $ (G, G_{0}). $

If $(G, \alpha, \beta, m, \varepsilon, i, G_{0} )$ is a groupoid
such that  $G_{0} \subseteq G$ and $\varepsilon:G_{0}\to G$ is the
inclusion, then we say that $(G, \alpha, \beta, m, i, G_{0} )$ is
a {\it $G_{0}-$groupoid}.

A {\it group bundle}, is a groupoid $ (G, G_{0}) $ such that $
\alpha(x) = \beta(x) $ for each $ x\in G. $

If $ (G, \alpha, \beta, G_{0})~$ is a groupoid, then $ Is(G) = \{
x\in G | \alpha(x) = \beta(x) \}$ is a group bundle, called the
{\it isotropy group bundle} of $ G.$

\markboth{Gheorghe Ivan}{On transitive group$-$groupoids}

Some basic properties of groupoids  follow directly from
Definition 2.1. For instance, if $ (G,\alpha, \beta , m, \e, i,
G_{0} ) $ is a groupoid, then (\cite{ivan99}):

$(i)~~~ \alpha(x y) = \alpha(x)~~~\hbox{and}~~~\beta(x
y)=\beta(y),~~~(\forall)~(x,y)\in G_{(2)};$

$(ii)~~ \alpha( x^{-1}) =\beta (x)~~~\hbox{and}~~~\beta(x^{-1})
=\alpha(x),~~ (\forall) x\in G;$

$(iii)~ \varepsilon(u)\cdot\varepsilon(u)=\varepsilon(u),~~~
(\varepsilon(u))^{-1}=\varepsilon(u), ~~~(\forall) u\in G_{0};$

$(iv)~~ $  if $~(x,y)\in G_{(2)},$ then $~ (y^{-1},x^{-1})\in
G_{(2)}~ $ and $~(x\cdot y)^{-1}=y^{-1}\cdot x^{-1};$

$(v)~~~ \alpha \circ \varepsilon= \beta \circ \varepsilon= i\circ
\varepsilon= Id_{G_{0}};$

$(vi)~~ \alpha \circ i = \beta,~~~ \beta \circ i = \alpha ~~
\hbox{and}~~ i \circ i = Id_{G}.$
\begin{defn}
{\rm (\cite{mack})
Let $ (G, G_{0}) $ and $ (G^{\prime}, G_{0}^{\prime}) $ be two
groupoids. A {\it morphism of groupoids} or {\it groupoid
morphism} from $ G $ into $ G^{\prime} $ is a pair $(f, f_{0}) $
of maps, where $ f:G \to G^{\prime} $ and $ f_{0}: G_{0} \to
G_{0}^{\prime} $ such that the following conditions hold:

$(1)~~~\alpha^{\prime}\circ f = f_{0} \circ
\alpha,~~~\hbox{and}~~~\beta^{\prime}\circ f = f_{0} \circ \beta
$;

$(2)~~~f(m(x,y)) = m^{\prime}(f(x),f(y))~$ for all $~(x,y)\in
G_{(2)}.$}\hfill$\Box$
\end{defn}

If $ G_{0} = G_{0}^{\prime} $ and $ f_{0} = Id_{G_{0}}, $ we say
that $ f:( G, G_{0}) \to ( G^{\prime}, G_{0})$ is a
 {\it $G_{0}-$morphism of groupoids}. A groupoid morphism $ (f, f_{0}): (G, G_{0}) \to (G^{\prime},
G_{0}^{\prime})$ such that $f$ and $f_{0}$ are bijective maps, is
called {\it isomorphism of groupoids}.

Note that if $ (f,f_{0}): (G; G_{0}) \to ( G^{\prime};
G_{0}^{\prime})~$ is a groupoid morphism then  $ f $ and $ f_{0} $
commute with all the structure functions of $ G $ and $
G^{\prime}.$ We have also the relations (\cite{ivan99}):
\begin{equation}
f \circ \varepsilon = \varepsilon^{\prime}  \circ
f_{0}~~~\hbox{and}~~~ f \circ i = i^{\prime}  \circ
f.\label{(2.1)}
\end{equation}
\begin{Ex}
{\rm $(i)~$ A nonempty set $ G_{0} $ may be regarded to be a
groupoid over $ G_{0}, $ called the {\it null groupoid} associated
to $ G_{0}. $ For this, we take $ \alpha = \beta = \varepsilon = i
= Id_{G_{0}} $ and $ u\cdot u = u $ for all $ u\in G_{0}. $

$(ii)~$ A group $ G $ having $ e $ as unity, is a  $ \{e
\}-$groupoid with respect to structure functions: $~\alpha (x) =
\beta (x): = e, ~G_{(2)}= G\times G,~ m (x,y):= xy,~
 \varepsilon:\{e\} \to G,~\varepsilon(e):= e $ and $ i:G \to G,
~i(x):= x^{-1}$. Conversely, a groupoid with one unit is a
group.}\hfill$\Box$
\end{Ex}
\begin{Rem}
{\rm  For the groupoids $ (G, \alpha, \beta, m, \e, i, G_{0}) $
and $ (G^{\prime}, \alpha^{\prime}, \beta^{\prime}, m^{\prime},
\e^{\prime}, i^{\prime}, G_{0}^{\prime}), $ one construct the
groupoid $
(G\times G^{\prime}, G_{0}\times G_{0}^{\prime})$ with the structure functions given by:\\
$~\alpha_{G\times G^{\prime}}(x,x^{\prime}) =
(\alpha(x),\alpha^{\prime}(x^{\prime}));~ ~~\beta_{G\times
G^{\prime}}(x,x^{\prime}) =
(\beta(x),\beta^{\prime}(x^{\prime}));$\\
$m_{G\times G^{\prime}}((x,y),(x^{\prime},y^{\prime}))= ( m(x,y),
m^{\prime}(x^{\prime}),
y^{\prime}),~(\forall)~(x, y)\in G_{(2)},~(x^{\prime},y^{\prime})\in G_{(2)}^{\prime};$\\
$~\e_{G\times G^{\prime}}(u,u^{\prime}) =
(\e(u),\e^{\prime}(u^{\prime})),~(\forall)~(u, u^{\prime})\in
G_{0}\times G_{0}^{\prime}~ $ and $ ~i_{G\times
G^{\prime}}(x,x^{\prime})=( i(x), i^{\prime}(x^{\prime}))$.

This groupoid is called the {\it direct product} of $ (G, G_{0}) $
and $ (G^{\prime}, G_{0}^{\prime})$.} \hfill$\Box$
\end{Rem}

 In the sequel we present the notion of group-groupoid as given in \cite{brspen}.
 Also we give some important properties of group-groupoids
 which are established in \cite{mivan}.

A group structure on a nonempty set $G$ is regarded as an
universal algebra determined by a binary operation $\omega $, a
nullary operation $\nu$ and an unary operation $\sigma$. For this
we use the notation $(G, \omega, \nu, \sigma )$ or $(G, \omega)$.
More precisely, the group operation is $\omega: G\times G \to
G,~(x,y) \to \omega (x,y):=x\oplus y$. The unit element of $G$ is
$e$; that is $\nu : \{\lambda\} \to G,~\lambda \to \nu(\lambda):=
e $ (here $\{\lambda\}$ is a singleton). The inverse of $x\in G$
is denoted by $\bar{x}$; that is $\sigma :G \to G,~x \to \sigma
(x):=\bar{x}$.

Let $ (G, \alpha, \beta, m, \varepsilon, i, G_{0})$ be a groupoid.
We suppose that on $G$ (resp., $G_{0}$) is defined a group
structure (the unit of $G$ (resp., $G_{0}$) is denoted by $e$
(resp., $e_{0}$)).

\begin{defn}
{\rm(\cite{brspen})  A {\it group-groupoid} or {\it $~{\cal
G}-$groupoid}, is a groupoid $ (G, \alpha, \beta, m, \varepsilon,
i, G_{0}) $ such that the following conditions hold:

$(i)~~ (G, \omega, \nu, \sigma)$ and $ (G_{0}, \omega_{0},
\nu_{0}, \sigma_{0})$ are groups.

$(ii)~~$ The maps $~(\omega, \omega_{0}): (G\times G, G_{0}\times
G_{0}) \to  (G,
  G_{0}),~~(\nu, \nu_{0}): (\{\lambda \}, \{\lambda \}) \to (G, G_{0})~$ and $~(\sigma, \sigma_{0}): (G, G_{0}) \to (G,
  G_{0})~$ are groupoid morphisms.}\hfill$\Box$
\end{defn}

We shall denote a group-groupoid by $ (G, \alpha, \beta, m, i, \e,
\oplus, G_{0})$.

\begin{Prop}
{\rm (\cite{mivan})} If  $~(G, \alpha, \beta, m, i, \e, \oplus,
G_{0})$ is a group-groupoid, then:

$(i)~~~$ the multiplication $m$ and binary operation $\omega$ are
compatible, that is:\\[-0.3cm]
\begin{equation}
(x\cdot y)\oplus (z\cdot t)= (x\oplus z)\cdot (y\oplus
t),~~~(\forall) (x,y), (z,t)\in G_{(2)};\label{(2.2)}
\end{equation}

$(ii)~~$ The source and target $ \alpha, \beta : G \to G_{0} $ are
surjective group morphisms; i.e., for all $ x,y\in G $
we have: \\[-0.3cm]
\begin{equation}
\alpha(x\oplus y)= \alpha(x)\oplus
\alpha(y),~~~\alpha(e)=e_{0},~~~\alpha(\bar{x})=\overline{\alpha(x)},\label{(2.3)}
\end{equation}
\begin{equation}
\beta(x\oplus y)= \beta(x)\oplus
\beta(y),~~~\beta(e)=e_{0},~~~\beta(\bar{x})=\overline{\beta(x)};\label{(2.4)}
\end{equation}

$(iii)~$ The inclusion map $ \varepsilon: G_{0} \to G $ is an
injective group morphism; i.e., for all $ u,v\in G_{0} $
we have:\\[-0.3cm]
\begin{equation}
\varepsilon(u\oplus v)= \varepsilon(u)\oplus \varepsilon(v),~~~
\varepsilon(e_{0})=e,~~~\varepsilon(\bar{u})=\overline{\varepsilon(u)};\label{(2.5)}
\end{equation}

 $(iv)~~$ The inversion $ i : G \to G $ is a group
automorphism; i.e., for all $ x,y\in G $
we have: \\[-0.3cm]
\begin{equation}
i(x\oplus y)= i(x)\oplus i(y),~~~i(e)=e,~~~
i(\bar{x})=\overline{i(x)}; \label{(2.6)}
\end{equation}

$(v)~~~$ The multiplication $m$ and the unary operation $\sigma$
are compatible, that is:\\[-0.3cm]
\begin{equation}
\sigma(x\cdot y)= \sigma(x)\cdot \sigma(y),~~~(\forall) (x,y)\in
G_{(2)}.\label{(2.7)}
\end{equation}
\end{Prop}
\begin{proof}
We present a short version of demonstration of this proposition
given in \cite{mivan}. For to prove the assertions $(i)-(iv)$ we
apply the fact that $(\omega, \omega_{0})$ is a groupoid morphism
and the relations (2.1).

$(i)~$ By hypothesis that  $ (\omega, \omega_{0})$ is a groupoid morphism, applying
 Definition 2.2(2), implies that:\\[0.1cm]
$(1)~~~\omega(m_{G\times G}((x,y),(z,t)))=m_{G}(\omega(x,z),
\omega(y,t)),~ (\forall)~(x,y), (z,t)\in G_{(2)}.$

We have\\[0.1cm]
$\omega(m_{G\times G}((x,y),(z,t)))=\omega(m_{G}(x,y),m_{G}(z,t))=
\omega(x\cdot y, z\cdot t) = (x\cdot y)\oplus (z\cdot t)~$
and\\[0.1cm]
$m_{G}(\omega(x,z), \omega(y,t))= m_{G}(x\oplus z, y\oplus t)=
(x\oplus z)\cdot (y\oplus t).$

Using $(1)$ one obtains $ (x\cdot y)\oplus (z\cdot t) = (x\oplus
z)\cdot (y\oplus t),$ and $(2.2)$ holds.

$(ii)~$ For each $(x,y)\in G\times G$, we have\\[0.1cm]
$\alpha (\omega(x,y)) =\alpha (x\oplus y)~$ and $~\omega_{0}(
(\alpha\times \alpha)(x,y))=\omega_{0}(\alpha(x), \alpha (y))=
\alpha(x)\oplus \alpha (y).$

 Using the equality $\alpha\circ \omega = \omega_{0}\circ
 (\alpha\times \alpha) $ , it follows $ \alpha (x\oplus y)=\alpha(x)\oplus \alpha (y),$ and
the first relation of $(2.3)$ holds. Then $\alpha $ is a morphism
of groups. As consequence of this assertion one obtains that the
two last relations from $(2.3)$ hold. Similarly, we prove that the
relations (2.4) hold.

$(iii)~$ and $~(iv)~$ Since $(\omega, \omega_{0})$ is a groupoid
morphism, it follows $~\omega\circ (\e\times \e)= \e \circ
\omega_{0} $ and $~ i\circ \omega = \omega\circ (i\times i).$
Applying the same procedure as in demonstration of $(ii)$, we
shall see that the relations $(2.5)$ and $(2.6)$ hold.

$(v)~$ Since $~(\sigma, \sigma_{0}) $ is a groupoid morphism,
for all $(x,y)\in G_{(2)}$ we have\\[0.1cm]
$\sigma(m(x,y))=m(\sigma(x), \sigma(y))$; i.e., $\sigma(x\cdot
y)=\sigma(x)\cdot \sigma(y).$
 Hence  $(2.7)$ holds.
\end{proof}

The relation $(2.2)$ (resp., $(2.7)$) is called the {\it
interchange law} between groupoid multiplication $m$ and group
operation $\omega$ (resp., $\sigma$).

 We say that the group-groupoid $ (G, \alpha, \beta, m, i, \e,
\oplus,  G_{0})$ is a {\it commutative group-groupoid}, if the
groups $G$ and $G_{0}$ are commutative.

Let us we give another definition for the notion of
group-groupoid. In the paper \cite{mivan} has proved that the
Definition $2.3$ and Definition $2.4$ are equivalent.
\begin{defn}
{\rm (\cite{mivan}) A {\it group-groupoid} is a groupoid $
(G,\alpha, \beta, m, i, \e, G_{0})$ such that the following
conditions are satisfied:

$(i)~~ (G, \oplus)$ and $ (G_{0}, \oplus )$ are groups;

$(ii)~~ $ the structure functions $\alpha, \beta : ( G, \oplus
)\to ( G_{0}, \oplus ),~\e : ( G_{0}, \oplus )\to ( G, \oplus )~$
and $~i : ( G, \oplus )\to ( G, \oplus )$ are morphisms of groups;

$(iii)~ $ the interchange law $(2.2)$ between  the operations $m$
and  $\oplus$ holds.} \hfill$\Box$
\end{defn}

If in Definition 2.4, we consider $ G_{0}\subseteq G$ and $
\varepsilon: G_{0}\to G$ is the inclusion map, then $ (G, \alpha,
\beta, m, i, \oplus, G_{0})$ is a group-groupoid. In this case, we
will say that {\it $(G, G_{0})$ is a group$-G_{0}-$groupoid}.

\begin{Ex}
{\rm $(i)~$ Let $G_{0}$ be a group. Then $ G_{0}$ has a structure
of null groupoid over $G_{0}$ (see Example 2.1(i)). We have that
$G= G_{0}$ and the maps $ \alpha, \beta,\varepsilon, i $ are
morphisms of groups. It is easy to prove that the interchange law
$(2.2)$ is verified. Then $G_{0}$ is a group-groupoid, called the
{\it null group-groupoid} associated to group $G_{0}$.

$(ii)~$ A commutative  group $(G, \oplus)$ having $\{e\}$ as unity
may be considered to be a $\{ e\}-$groupoid (see Example 2.1(ii)).
In this case we have that the groupoid multiplication and the
group operation coincide, that is $m =\oplus$. We have that $(G,
\oplus)$ and $G_{0}=\{e\}$ are groups. It is easy to see that
$\alpha, \beta, \e $ and $i$ are morphisms of groups. It remains
to verify that the interchange law $(2.2)$ holds. Indeed, for
$x,y,z,t \in G$ we have $(x\oplus y)\oplus (z\oplus t)= (x\oplus z
)\oplus (y \oplus t),$ since the operation $\oplus$ is associative
and commutative. Hence $(G, \alpha, \beta, m, \e, i, \oplus,
\{e\})$ is a group-groupoid called the {\it group-groupoid with a
single unit} associated to commutative group $(G, \oplus)$.
 Therefore, {\it each commutative group $ G $ can be regarded as a commutative group
 $-\{e\}-$groupoid}.}\hfill$\Box$
\end{Ex}
\begin{defn}
{\rm Let $ ( G_{j}, \alpha_{j}, \beta_{j}, m_{j}, \varepsilon_{j},
i_{j}, \oplus_{j}, G_{j,0} ),~ j=1,2$ be two group-groupoids. A
groupoid morphism  $ (f, f_{0}): (G_{1}, G_{1,0})\to (G_{2},
G_{2,0}) $ such  that $ f $ and $ f_{0}$ are group morphisms, is
called  {\it group-groupoid morphism} or {\it morphism of
group-groupoids}.

A group-groupoid morphism of the form $ (f, Id_{G_{1,0}}): (G_{1},
G_{1,0})\to (G_{2}, G_{1,0}) $ is called  {\it $ G_{1,0}-$morphism
of group-groupoids}.  It is denoted by $ f: G_{1} \to
G_{2}$.}\hfill$\Box$
\end{defn}

\section{ Transitive group-groupoids}

Let $(G, \alpha, \beta, m, \e, i, \oplus, G_{0})$ be a
group-groupoid. We say that it is {\it transitive}, if its anchor
map is surjective.

\begin{Ex}
 {\rm  The Cartesian product $ G:= X \times X~$ has a structure
of groupoid over the set $ X $  by taking the structure functions
as follows: $ \overline{\alpha}(x,y):= x,~\overline{\beta}(x,y):=
y;~ (x,y) $ and $ (y^{\prime},z) $ are composable
 iff $ y^{\prime} = y $ and we define
$ (x,y)\cdot (y,z):=(x,z) $; the inclusion map $
\overline{\varepsilon}:X \to X\times X$ is given by
$\overline{\varepsilon}(x):=(x,x)$ and the inverse of $~(x,y)~$ is
defined by $ (x,y)^{-1}:= (y,x). $ This is called the {\it pair
groupoid} associated to set $X$. Its unit set is $ \e(X)=\{
(x,x)\in X\times X | x\in X\}.$

We suppose that $(X,\oplus)$ is a group. Then  $ X\times X$ is a
group endowed with operation  $ (x_{1}, x_{2}) \oplus (y_{1},
y_{2}):=( x_{1}\oplus y_{1}, x_{2}\oplus y_{2}),~$ for all $
x_{1}, x_{2}, y_{1}, y_{2}\in G$. It is easy to check that
$~\overline{\alpha}, \overline{\beta},~\overline{\varepsilon} $
and $~\overline{i}~$ are group morphisms. For to prove that the
condition $(iii)$ from Definition 2.4 is verified, we consider $
x=(x_{1}, x_{2}), y=(y_{1}, y_{2}), z=(z_{1}, z_{2}), t=(t_{1},
t_{2}) $ from $ X\times X$ such that
$\overline{\beta}(x)=\overline{\alpha}(y) $ and
$\overline{\beta}(z)= \overline{\alpha}(t).$ Then $ x_{2}=y_{1} $
and $ z_{2}=t_{1}$. It follows $y=(x_{2}, y_{2}),~t=(z_{2},
t_{2}), x\cdot y = (x_{1}, y_{2}) $ and $ z\cdot t = (z_{1},
t_{2}). $ We have $~ (x\cdot y)\oplus (z\cdot t) = (x_{1},
y_{2})\oplus (z_{1},
t_{2})= (x_{1}\oplus z_{1}, y_{2}\oplus t_{2})~ $ and\\[0.1cm]
$ (x\oplus z)\cdot (y\oplus t)=  (x_{1}\oplus z_{1}, x_{2}\oplus
z_{2})\cdot (x_{2}\oplus z_{2}, y_{2}\oplus t_{2})=(x_{1}\oplus
z_{1}, y_{2}\oplus t_{2}).~$ Then, $ (x\cdot y)\oplus (z\cdot
t)=(x\oplus z)\cdot (y\oplus t)$ and so the relation $(2.2)$
holds. Hence $X\times X$ is a group-groupoid called the  {\it
group-pair groupoid} associated to group $X$.}
 \hfill$\b$
\end{Ex}
\begin{Prop}
The group-pair groupoid $(X\times X, \overline{\alpha},
\overline{\beta}, \overline{m},\overline{\e},\overline{i}, X)$
associated to group $(X, \oplus)$ is transitive.
\end{Prop}
\begin{proof}
The anchor  $(\overline{\alpha} ,\overline{\beta}): X\times X \to
X\times X $ is surjective, since for $(u, v)\in X\times X$ there
exists $x =(u,v)\in X\times X $ such that $ (\overline{\alpha},
\overline{\beta})(x)= (\overline{\alpha}(u,v),
\overline{\beta}(u,v))=(u, v).$
\end{proof}
\begin{Ex}
{\rm {\bf The modular group-groupoid ${\bf Z}_{n}^{2}(a)$}. Let
${\bf Z}_{n}$ be the additive group of integers modulo $n$ and let
$a\in {\bf Z}_{n}$ such that $ a^{2} = 1$.

We consider the sets  $ G:= {\bf Z}_{n}\times {\bf Z}_{n}$ and
$G_{0}:= \{ (x, a x) \in {\bf Z}_{n}\times {\bf Z}_{n} | (\forall)
x\in {\bf Z}_{n}\}$. We have that $G$ is a commutative group and
$G_{0}\subseteq G$ is a subgroup.

We have $G_{0}:= \{(x, a x)\in {\bf Z}_{n}\times {\bf Z}_{n} |
(\forall) x\in {\bf Z}_{n} \}= \{(ay, y)\in {\bf Z}_{n}\times {\bf
Z}_{n} | (\forall) y\in {\bf Z}_{n}\}$.

The functions  $\alpha,\beta: G\longrightarrow G_{0}, ~i:
G\longrightarrow G$ and $m : G_{(2)} \to G$ where\\
 $G_{(2)}:=\{((x,y),(ay,z))\in G\times G~|~x,y,z\in {\bf
Z}_{n}\}\longrightarrow G$ are given by:\\[-0.3cm]
\[
\alpha(x,y):=(x,ax);~~ \beta(x,y):=(ay,y);~~i(x,y):=(ay,ax);~~
(x,y)\cdot(ay,z):=(x,z).
\]

It is easy to verify that $ \alpha, \beta, m, i $ determine on $G$
a structure of a $G_{0}-$groupoid. Also, $\alpha, \beta$ and $i$
are group morphisms. Therefore, the conditions $(i)$ and $(ii)$
from the Definition 2.4 hold.

We consider $x,y, z,t \in G$ such that $(x,y), (z,t)\in G_{(2)}.$
Then $ x=(x_{0}, x_{1}),~ y=(ax_{1}, y_{1}),~
z=(z_{0}, z_{1})$ and $ t=(a z_{1}, t_{1}).$ We have\\[0.1cm]
$(x\cdot y) + (z\cdot t) = (x_{0}, y_{1}) + ( z_{0}, t_{1}) =
(x_{0}+
z_{0}, y_{1}+ t_{1})~$ and\\[0.1cm]
$(x + z)\cdot (y+t) = (x_{0}+ z_{0}, x_{1}+ z_{1})\cdot ( a(x_{1}+
z_{1}),y_{1}+ t_{1})= (x_{0}+ z_{0}, y_{1}+ t_{1}).$

Hence, $(x\cdot y) + (z\cdot t) = (x + y)\cdot (z+t)$ and the
interchange law $(2.2)$ holds. Therefore, $(G, \alpha, \beta, m,
i, G_{0})$ is a commutative group$-G_{0}-$groupoid, called the
{\it modular group-groupoid of type $a$} and it is denoted by $
{\bf Z}_{n}^{2}(a)$.

 The group-groupoid $ {\bf
Z}_{n}^{2}(a)$ is transitive. Indeed, the anchor $(\alpha,
\beta):G\rightarrow G_{0}\times G_{0}$ is given by $ (\alpha,
\beta)(x,y):=((x, ax), (ay, y)),~(\forall) x, y\in
\mathbb{Z}_{n}.$ Let $(u,v)\in G_{0}\times G_{0}$ with $ u=(x_{0},
a x_{0})$ and $v=(y_{0}, ay_{0})$. Taking $x=(x_{0}, ay_{0})\in
G$, we have $(\alpha, \beta )(x)=(\alpha, \beta )(x_{0}, a
y_{0})=((x_{0}, a x_{0}),( a^{2} y_{0}, a y_{0}))=(u,v).$ Hence,
$(\alpha, \beta)$ is surjective .}\hfill$\Box$
\end{Ex}

Let us we shall give some ways of building up new transitive
group-groupoids.

{\bf $1.~$ Direct product of two group-groupoids}. Let given the
group-groupoids $ (G, \alpha, \beta, m, \e, i, \oplus_{G}, G_{0})
$ and $ (G^{\prime}, \alpha^{\prime}, \beta^{\prime}, m^{\prime},
\e^{\prime}, i^{\prime}, \oplus_{G^{\prime}}, G_{0}^{\prime}).$
Consider the direct product  $ (G\times G^{\prime},
\alpha_{G\times G^{\prime}}, \beta_{G\times G^{\prime}},
m_{G\times G^{\prime}}, \e_{G\times G^{\prime}}, i_{G\times
G^{\prime}}, G_{0}\times G_{0}^{\prime} )~$ of the groupoids $(G,
G_{0}) $ and $ (G^{\prime}, G_{0}^{\prime})$ (see Remark 2.1). On
$ G\times G^{\prime}$ and $ G_{0}\times G_{0}^{\prime}$ we
introduce the usual group operations. These operations are defined by\\
$(x, x^{\prime})\oplus_{G\times G^{\prime}} (y, y^{\prime}):= ( x
\oplus_{G} y, x^{\prime} \oplus_{G^{\prime}} y^{\prime} ), ~ (\forall) x, y\in G,~ x^{\prime}, y^{\prime}\in G^{\prime}$  ~~\hbox{and}\\
$(u_{1}, u_{1}^{\prime})\oplus_{G_{0}\times G_{0}^{\prime}}
(u_{2}, u_{2}^{\prime}):= ( u_{1} \oplus_{G_{0}} u_{2},
u_{1}^{\prime} \oplus_{G_{0}^{\prime}} u_{2}^{\prime} ), ~
(\forall) u_{1}, u_{2}\in G_{0},~ u_{1}^{\prime},
u_{2}^{\prime}\in G_{0}^{\prime}$.

By a direct computation we prove that the conditions from
Definition 2.4 are satisfied. Then  $ (G\times G^{\prime},
\alpha_{G\times G^{\prime}}, \beta_{G\times G^{\prime}},
m_{G\times G^{\prime}}, \e_{G\times G_{0}^{\prime}}, i_{G\times
G^{\prime}}, \oplus_{G\times G^{\prime}}, G_{0}\times
G_{0}^{\prime})$ is a group-groupoid, called the {\it direct
product of group- groupoids} $(G,G_{0})$ and
$(G^{\prime},G_{0}^{\prime})$.

The projections $ pr_{G} : G\times G^{\prime} \to G$ and $
pr_{G^{\prime}} : G\times G^{\prime} \to G^{\prime}$ are morphisms
of group-groupoids, called the {\it canonical projections} of  $
G\times G^{\prime}$ onto $G$ and $ G^{\prime}$, respectively.

It is easy to prove that: {\it  the direct product of two
transitive group-groupoids is also a transitive
group-groupoid}.\hfill$\Box$

{\bf $2.~$ The group-groupoid associated to an epimorphism of
commutative groups}. Let $(E, +)$ and  $(F, +)$ be two commutative
groups, and $\pi :E \to F$ be an group epimorphism.

The set $ G_{\pi} = E\times_{\pi}E = \{(x,y)\in E\times E | \pi(x)
= \pi(y)\} $ has a structure of groupoid over $ E $ for which the
structure functions $ \alpha_{\pi}, \beta_{\pi}, m_{\pi}, \e_{\pi}
$ and $ i_{\pi} $ are the restrictions to $ G_{\pi} $ of the
structure functions of the pair groupoid $ E\times E $ over $ E.$
More precisely, the maps $\alpha_{\pi}, \beta _{\pi}: G_{\pi} \to
E,~m_{\pi}:(G_{\pi})_{(2)}\to E,~ \varepsilon_{\pi}: E\to
G_{\pi}~$ and $~i _{\pi}: G_{\pi}\to G_{\pi}$ are defined by
$~\alpha_{\pi}(x,y):= x,~~\beta_{\pi} (x,y):=
y,~~\varepsilon_{\pi}(x):= (x,x),~~m_{\pi}((x,y),(y,z)):= (x,z)~$
and $i_{\pi}(x,y):= (y,x),$ for all $x, y, z\in E.$

Clearly, $G_{\pi}$ has a structure of commutative group determined
by the group operation $+_{\pi}: G_{\pi}\times G_{\pi} \to
G_{\pi}$ given by $(x_{1}, y_{1})+_{\pi}(x_{2}, y_{2})= (x_{1}+
x_{2}, y_{1}+y_{2})$ for all $ (x_{1}, y_{1}), (x_{2}, y_{2})\in
G_{\pi}.$

It is easy to check that $(G_{\pi}, \alpha_{\pi}, \beta_{\pi}, m
_{\pi}, \varepsilon_{\pi}, i_{\pi}, E)$ verifies the conditions
from Definition 2.4. Then $G_{\pi}$ is a commutative
group-groupoid over $E$, called the {\it group-groupoid associated
to $\pi: E \to F.$} Also, it is a transitive.\hfill$\Box$

{\bf $3.~$ Trivial group-groupoid  ${\cal TGG}(A,B) $}. Let
$(A,\oplus_{A})$ and $(B,\oplus_{B})$ be two commutative groups.
The unit element of the group $A$ (resp., $B$) is denoted by
$e_{A}$ (resp., $e_{B}$). The inverse of $a\in A$
 (resp. $b\in B$) is denoted by $\bar{a}_{A}$ (resp., $\bar{b}_{B}$).

The Cartesian product $ B\times A \times B $ has a natural
structure of group. We introduce on $ {\cal G}:= B\times A \times
B $ the structure functions $ \alpha_{\cal G}, \beta_{\cal G},
m_{\cal G}, \varepsilon_{\cal G} $ and $ i_{\cal G}$ as follows.

For all $(b_{1}, a, b_{2})\in {\cal G} $ and $b\in B$, the maps $
\alpha_{\cal G},~\beta_{\cal G}: {\cal G} \to B,~
\varepsilon_{\cal G}: B \to {\cal G} $ and $ i_{\cal G}: {\cal
G}\to {\cal G}$ are defined by: $~\alpha_{\cal G}(b_{1}, a,
b_{2}):= b_{1},~  \beta_{\cal G}(b_{1}, a, b_{2}):= b_{2},~
\varepsilon_{\cal G}(b)= (b, e_{A}, b),~ i_{\cal G}(b_{1}, a,
b_{2}):=( b_{2}, \overline{a}_{A}, b_{1} )$. Then $ {\cal G}_{(2)}
= \{ ( (b_{1}, a_{1}, b_{2}), (b_{2}^{\prime}, a_{2}, b_{3}) ) \in
{\cal G}\times {\cal G} |  b_{2} = b_{2}^{\prime} \}$ and the
partially  multiplication $m_{\cal G}: {\cal G}_{(2)} \to {\cal G}
$ is given by\\[0.1cm]
$~~~~~~~~~~~~~~~~~~(b_{1}, a_{1}, b_{2})\cdot_{\cal G} (b_{2},
a_{2}, b_{3}):= ( b_{1}, a_{1}\oplus_{A} a_{2}, b_{3})$.

 It is easy to verify that the conditions of Definition
2.1 are satisfied. Then\\
 $( {\cal G}, \alpha_{\cal G},
\beta_{\cal G}, m_{\cal G}, \varepsilon _{\cal G}, i_{\cal G}, B)
$ is a groupoid. Also, the condition $(i)$ from Definition 2.4 is
 verified.

 Let now two elements $ x,y\in {\cal G}$,
 where $ x = (b_{1}, a_{1}, b_{2}) $ and $ y = (b_{3}, a_{2}, b_{4})
 $. We have\\[0.1cm]
$\alpha_{\cal G}(x \oplus_{\cal G} y)= \alpha_{\cal
G}(b_{1}\oplus_{B} b_{3}, a_{1}\oplus_{A} a_{2}, b_{2}\oplus_{B}
b_{4})= b_{1}\oplus_{B} b_{3}= \alpha_{\cal G}(x) \oplus_{B}
\alpha_{\cal G}(y).$ Then $ \alpha_{\cal G}$ is a group morphism.
Similarly we prove that $ \beta_{\cal G}$ is a group morphism.

For all $b_{1}, b_{2}\in B$ we have $ \e_{\cal G}(b_{1} \oplus_{B}
b_{2})= ( b_{1} \oplus_{B} b_{2}, e_{A} , b_{1} \oplus_{B}
b_{2})~$ and\\[0.1cm]
$\e_{\cal G}(b_{1})\oplus_{\cal G} \e_{\cal G}(b_{2})= (b_{1},
e_{A}, b_{1})\oplus_{\cal G} (b_{2}, e_{A}, b_{2})= ( b_{1}
\oplus_{B} b_{2}, e_{A} , b_{1} \oplus_{B} b_{2}).$  It follows
that $ \e_{\cal G}$ is a group morphism.

For $ x = (b_{1}, a_{1}, b_{2})\in {\cal G} $ and $ y = (b_{3},
a_{2}, b_{4}) \in {\cal G}$, we have\\[0.1cm]
$ i_{\cal G}(x \oplus_{\cal G} y)= i_{\cal G}( b_{1}\oplus_{B}
b_{3}, a_{1}\oplus_{A} a_{2}, b_{2}\oplus_{B} b_{4})= (
b_{2}\oplus_{B} b_{4}, \overline{a_{1}\oplus_{A} a_{2}},
b_{1}\oplus_{B} b_{3}) =$\\[0.1cm]
$=( b_{2},\overline{a_{1}}, b_{1}) \oplus_{\cal G} ( b_{4},
\overline{a_{2}}, b_{3})= i_{\cal G}(x) \oplus_{\cal G} i_{\cal
G}(y). $ It follows that $ i_{\cal G}$ is a group morphism.

Hence the conditions (ii) from Definition 2.4 hold.

For to verify the interchange law $ (2.2))$, we consider $ x, y,
z, t \in {\cal G} $ such that\\ $(x,y)\in {\cal G}_{(2)} $ and
$(z,t)\in {\cal G}_{(2)} $. Then $ x = (b_{1}, a_{1}, b_{2} ), y =
( b_{2}, a_{2}, b_{3}), z = (b_{4}, a_{3}, b_{5}) $ and $ t =
(b_{5}, a_{4}, b_{6}) $. We have

 $(1)~~~ (x\cdot_{\cal G} y) \oplus_{\cal G}(z \cdot_{\cal G} t)=
 (b_{1}, a_{1}\oplus_{A} a_{2}, b_{3})\oplus_{\cal G}(b_{4}, a_{3}\oplus_{A} a_{4},
 b_{6})=$\\[0.1cm]
 $ ( b_{1}\oplus_{B} b_{4}, a_{1}\oplus_{A} a_{2}\oplus_{A} a_{3}\oplus_{A}
 a_{4}, b_{3}\oplus_{B} b_{6})~$ and

$(2)~~~ (x\oplus_{\cal G} z)\cdot_{\cal G}(y \oplus_{\cal G} t)= (
b_{1}\oplus_{B} b_{4}, a_{1}\oplus_{A} a_{3}, b_{2}\oplus_{B}
b_{5} )\cdot_{\cal G} ( b_{2}\oplus_{B} b_{5}, a_{2}\oplus_{A}
a_{4}, b_{3}\oplus_{B} b_{6})=$\\[0.1cm]
$= ( b_{1}\oplus_{B} b_{4}, a_{1}\oplus_{A} a_{3}\oplus_{A}
a_{2}\oplus_{A} a_{4}, b_{3}\oplus_{B} b_{6}).$

Using $(1)$ and $(2)$ we obtain $(x\cdot_{\cal G} y) \oplus_{\cal
G}(z \cdot_{\cal G} t)= (x\oplus_{\cal G} z)\cdot_{\cal G}(y
\oplus_{\cal G} t),$ since the operation $\oplus_{A}$ is
commutative, and the relation $(2.2)$ holds. Hence $ {\cal G}: =
B\times A\times B $ is a commutative group-groupoid over $B$. Its
set of units is $\e(B)=\{(b,e_{A},b)\in {\cal G} | b\in B\}.$

The commutative group-groupoid $ ( {\cal G}:= B\times A\times B,
\alpha_{\cal G}, \beta_{\cal G}, \cdot_{\cal G}, \e_{\cal G},
i_{\cal G}, \oplus_{\cal G}, B) $ is called the {\it trivial
group-groupoid} associated to pair of commutative groups  $(A, B)$
. This commutative group-groupoid is denoted by $ {\cal
TGG}(A,B)$. The isotropy group at $b\in B$ is $~{\cal G}(b) = \{
(b,a,b) | a\in A\} $ which identify with the group $( A,
\oplus_{A} )$.

\begin{Prop}
Let $(A,B)$ and $(A^{\prime}, B^{\prime})$ be two pairs of
commutative groups and $ \theta: A \to A^{\prime}~$ and
$~\theta_{0}: B \to B^{\prime}$ be two group morphisms. Then:

$(i)~$ the trivial group-groupoid $(B\times A\times B, B)$ is
transitive.

$(ii)~ (\theta_{0} \times \theta \times \theta_{0}, \theta_{0}):
(B\times A\times B, B ) \to (B^{\prime}\times A^{\prime}\times
B^{\prime}, B^{\prime}),~$ is a morphism of group-groupoids.
\end{Prop}
\begin{proof}
We denote ${\cal G}:= B\times A\times B, {\cal G}^{\prime}:=
B^{\prime}\times A^{\prime}\times B^{\prime}$ and $
\widetilde{\theta}:=\theta_{0} \times \theta \times \theta_{0}. $

 $(i)~$ The anchor  $(\alpha_{\cal G} ,\beta_{G}): {\cal
G}\longrightarrow B\times B$ is a surjective map. Indeed, for
$(b_{1}, b_{2})\in B\times B$ there exists $x =(b_{1}, a,b_{2})\in
{\cal G}$ such that $ (\alpha_{\cal G} ,\beta_{G})(x)=
(\alpha_{\cal G}(x),\beta_{\cal G}(x))= (b_{1}, b_{2}). $

$(ii)~$ We first prove that the pair  $(\widetilde{\theta},
\theta_{0}) $ verifies the conditions from Definition 2.2.
 For all $x=(b_{1}, a, b_{2})\in {\cal G}$, we have
$~(\alpha_{{\cal G}^{\prime}}\circ \widetilde{\theta})(x)=
\alpha_{{\cal G}^{\prime}}(\theta_{0}(b_{1}), \theta (a),
\theta_{0}(b_{2}))=\theta_{0}(b_{1})~$ and $~(\theta_{0}\circ
\alpha_{\cal G})(x)=\theta_{0}(\alpha_{\cal G}(b_{1}, a,
b_{2}))=\theta_{0}(b_{1}).$ Then $\alpha_{{\cal G}^{\prime}}\circ
\widetilde{\theta}=\theta_{0}\circ \alpha_{\cal G}.$ Similarly, we
prove that $\beta_{{\cal G}^{\prime}}\circ
\widetilde{\theta}=\theta_{0}\circ \beta_{\cal G}.$

We consider $x, y\in {\cal G}$ such that $(x,y)\in {\cal G}_{(2)}.
$  Then $x=(b_{1}, a_{1}, b_{2}), y=(b_{2}, a_{2}, b_{3}) $ and $
x\cdot_{\cal G} y = (b_{1}, a_{1}\oplus_{A} a_{2}, b_{3}).$ We
have $~\widetilde{\theta}(x\cdot_{\cal G} y)= (\theta_{0}(b_{1}),
\theta(a_{1}\oplus_{A} a_{2}), \theta_{0}(b_{3})) ~$ and\\[0.1cm]
$\widetilde{\theta}(x)\cdot_{\cal G^{\prime}} \widetilde{\theta}(
y)= (\theta_{0}(b_{1}), \theta(a_{1})\oplus_{A^{\prime}} \theta
(a_{2}), \theta_{0}(b_{3})).$ It follows that
$\widetilde{\theta}(x\cdot_{\cal G} y)=
\widetilde{\theta}(x)\cdot_{\cal G^{\prime}} \widetilde{\theta}(
y),$ since $ \theta(a_{1}\oplus_{A} a_{2}) =
\theta(a_{1})\oplus_{A^{\prime}} \theta(a_{2})$. Hence,
$(\widetilde{\theta}, \theta_{0}) $ is a groupoid morphism.

Since $\theta$ and $\theta_{0}$ are morphisms of groups it implies
that $\widetilde{\theta}: {\cal G} \to {\cal G}^{\prime}$ is a
group morphism. Therefore, the conditions from Definition 2.5 are
satisfied. Hence, $(\widetilde{\theta}, \theta_{0})$ is a morphism
of group-groupoids.
\end{proof}
\begin{Th}
Let $(G, \alpha, \beta, m, \e, i, \oplus, G_{0})$ be a transitive
commutative group-groupoid. If  the restriction $\beta_{e_{0}}$
 of the target map $\beta$ to $G^{e_{0}}= \alpha^{-1}(e_{0})$ is a split epimorphism of groups, then $ (G, G_{0})$
 is isomorphic to trivial group-groupoid ${\cal TGG}(G(e_{0}), G_{0}) $ associated to pair $(G(e_{0}), G_{0}).$
\end{Th}
\begin{proof}
Since the anchor $(\alpha, \beta): G \to G_{0}\times G_{0}$ is
surjective, it follows that the restriction $\beta_{e_{0}}$ of
$\beta :G\to G_{0}$ to $G^{e_{0}}= \{a\in G | \alpha(a)=e_{0}\}$
is surjective. Let $\gamma : G_{0} \to G^{e_{0}}$ be any section
of $\beta_{e_{0}}: G^{e_{0}} \to G_{0}$; i.e., $\gamma $ is an
injective map such that $\beta_{e_{0}}\circ \gamma = Id_{G_{0}}.$
Then  $\beta(\gamma (v))= v$ for all $v\in G_{0}$. From hypothesis
that $\beta_{e_{0}}$ is a split epimorphism it implies that
$\gamma$ is a morphism between the commutative groups $G_{0}$ and
$G^{e_{0}}$. Then\\[-0.43cm]
\begin{equation}
\gamma(u\oplus v) = \gamma(u)\oplus \gamma(v),~~~ (\forall) u,v\in
G_{0}.\label{(3.1)}\\[-0.1cm]
\end{equation}

Consider the isotropy group $G(e_{0})=\{x\in G
|\alpha(x)=\beta(x)=e_{0}\}$ and the trivial group-groupoid
$(G_{0}\times G(e_{0})\times G_{0}, G_{0})$ associated to pair
$(G(e_{0}), G_{0}).$

Finally, define the map $ \varphi :G\rightarrow G_{0}\times
G(e_{0})\times G_{0}$ by the formula\\[-0.3cm]
\begin{equation}
\varphi(x) = (\alpha(x), \gamma (\alpha(x)) \cdot x\cdot \gamma
(\beta(x))^{-1}, \beta(x)), ~~~(\forall) x\in G.
\label{(3.2)}\\[-0.1cm]
\end{equation}

The product $ \gamma (\alpha(x)))\cdot x $ is defined, since $
\beta (\gamma (\alpha(x))))=\alpha(x)$. Also, $~x\cdot \gamma
(\beta(x))^{-1} $ is defined, since $
\alpha(\gamma(\beta(x))^{-1}) = \beta
(\gamma(\beta(x)))=\beta(x)$. Then $ y:=\gamma (\alpha(x)))\cdot
x\cdot \gamma (\beta(x))^{-1}$ is defined in $G$. We have $ y\in
G(e_{0})$. Indeed, $ \alpha(y)= \alpha(\gamma (\alpha(x))))= e_{0}
$ and $ \beta(y)= \beta (\gamma (\beta(x))^{-1}))= \alpha (\gamma
(\beta(x))= e_{0},$ since $ \gamma(\alpha(x)), \gamma
(\beta(x))\in G^{e_{0}}$. Hence, the map $\varphi$ is
well-defined.

Denoting $H:= G_{0}\times G(e_{0})\times G_{0}$, we shall prove
that $(\varphi, Id_{G_{0}}) : (G, G_{0}) \to (H, G_{0})$ is a
morphism of groupoids. For all $x\in G$, we have:\\[0.1cm]
$(\alpha_{H}\circ \varphi)(x)= \alpha_{H}(\alpha(x), \gamma
(\alpha(x)) \cdot x\cdot \gamma (\beta(x))^{-1},
\beta(x))=\alpha(x). $

Hence, $\alpha_{H}\circ \varphi = \alpha.$ Similarly, we verify
that $\beta_{H}\circ \varphi = \beta.$

Let $x,y\in G$ such that $(x,y)\in G_{(2)}$. Denote $z:=x\cdot y.$
Then $ \beta(x)=\alpha(y),$\\
$\alpha(z)=\alpha(x), \beta(z)=\beta(y) $ and we have \\[0.1cm]
$\varphi(z)=(\alpha(z), \gamma (\alpha(z))\cdot z \cdot \gamma
(\beta(z))^{-1}, \beta(z))=(\alpha(x), \gamma (\alpha(x)) \cdot
(x\cdot y)\cdot \gamma (\beta(y))^{-1}, \beta(y)) $ and
$~\varphi(x)\cdot_{H} \varphi(y)= (\alpha(x),\gamma
(\alpha(x))\cdot x\cdot \gamma (\alpha(y))^{-1}\cdot \gamma
(\alpha(y)) \cdot y\cdot \gamma
(\beta(y))^{-1}, \beta(y))=$\\[0.1cm]
$= (\alpha(x),\gamma (\alpha(x)) \cdot x\cdot y\cdot \gamma
(\beta(y))^{-1}, \beta(y)).$

It follows that $\varphi(x\cdot y) =\varphi(x)\cdot_{H}
\varphi(y).$ Therefore, $(\varphi, Id_{G_{0}})$ is a groupoid
morphism.

We prove that $\varphi :(G,\oplus)\to (H, \oplus_{H})$ is morphism
of groups. Applying $(3.1)$ and the fact that $\alpha$ and $\beta$
are morphisms of groups, for all $x,y\in G$  we have\\[0.1cm]
$\varphi(x\oplus y)= ( \alpha(x)\oplus \alpha(y),
\gamma(\alpha(x)\oplus \alpha(y))\cdot (x\oplus y)\cdot \gamma
(\beta(x)\oplus \beta(y))^{-1}, \beta(x)\oplus \beta(y)) =$\\[0.1cm]
$=( \alpha(x)\oplus \alpha(y), (\gamma(\alpha(x))\oplus
\gamma(\alpha(y)))\cdot (x\oplus y)\cdot  (\gamma (\beta(x)\oplus
\gamma(\beta(y))^{-1}, \beta(x)\oplus \beta(y))$.

Using the notations $ a:=\gamma(\alpha(x)),~b:=\gamma(\beta(x)),~
c:=\gamma(\alpha(y)) ~$ and $~ d:=\gamma(\beta(y))$ one obtains $
\varphi(x\oplus y)=( \alpha(x)\oplus \alpha(y), (a\oplus c)\cdot
(x\oplus y)\cdot (b \oplus d)^{-1}, \beta(x)\oplus \beta(y))$.

Applying $(2.6)$ we deduce that $(b \oplus d)^{-1}= b^{-1} \oplus
d^{-1}.$ Then\\[0.1cm]
$(1)~~~~~\varphi(x\oplus y)=( \alpha(x)\oplus \alpha(y), (a\oplus
c)\cdot (x\oplus y)\cdot (b^{-1} \oplus d^{-1}), \beta(x)\oplus
\beta(y))$.

Also, we have $~\varphi(x)= (\alpha(x), a\cdot x\cdot b^{-1},
\beta(x)),~\varphi(y)=(\alpha(y), c\cdot y\cdot d^{-1},
\beta(y))~$ and\\[0.1cm]
$(2)~~~~~\varphi(x)\oplus_{H} \varphi(y)= (\alpha(x)\oplus
\alpha(y), (a\cdot x\cdot b^{-1})\oplus (c\cdot y\cdot
d^{-1}),\beta(x)\oplus \beta(y)).$

Applying $(2.2)$ one obtains\\
$(3)~~~~~(a\cdot x\cdot b^{-1})\oplus (c\cdot y\cdot d^{-1})=
(a\oplus c)\cdot (x\oplus y)\cdot (b^{-1}\oplus d^{-1})$.

From the relations $(1)-(3)$ it follows $~\varphi(x\oplus
y)=\varphi(x)\oplus_{H} \varphi(y), $ and $\varphi $ is a morphism
of groups. Hence $(\varphi, Id_{G_{0}})$ is a morphism of
group-groupoids.

The map $\varphi$ is injective. For this, let $x,y\in G$ such that
$\varphi(x)=\varphi(y).$  Then $(\alpha(x), a\cdot x\cdot b^{-1},
\beta(x))=(\alpha(y), c\cdot y\cdot d^{-1}, \beta(y)) $ and it
implies that $\alpha(x) = \alpha(y),$\\
 $a\cdot x\cdot
b^{-1}=c\cdot y\cdot d^{-1}~$ and $~\beta(x)=\beta(y).$ We have
$a=c$ and $b=d.$ Applying the simplification law in a groupoid,
from $a\cdot x\cdot b^{-1}=a\cdot y\cdot b^{-1}~$ one obtains $
x=y.$

The map $\varphi$ is surjective. For this, let $(u, z, v)\in H $.
Then $u,v\in G_{0}$ and $z\in G(e_{0}).$ We consider the element
$x_{0}:= \gamma(u)^{-1}\cdot z\cdot \gamma(v).$ The product
$\gamma(u)^{-1}\cdot z$ is defined, since $\beta(\gamma(u)^{-1}) =
\alpha(\gamma(u))= e_{0}=\alpha(z)$. Also, $ z\cdot \gamma(v)$ is
defined, since $\alpha(\gamma(v))=e_{0}$ and $\beta(z)=e_{0}$.
Hence, $x_{0}$ is defined and $x_{0}\in G$. We have that
$\varphi(x_{0}) = (u, z, v).$ Indeed, using $(3.2)$ it follows
$\varphi(x_{0}) = (\alpha(x_{0}), \gamma(\alpha(x_{0}))\cdot
x_{0}\cdot \gamma(\beta(x_{0}))^{-1}, \beta(x_{0})).$ We have\\
 $\alpha(x_{0})= \alpha(\gamma(u)^{-1})=\ \beta(\gamma(u))= u,~
\beta(x_{0})= \beta(\gamma(v))=v $ and\\
 $\gamma(\alpha(x_{0}))\cdot x_{0}\cdot \gamma(\beta(x_{0}))^{-1}=
\gamma(u)\cdot (\gamma(u)^{-1}\cdot z\cdot \gamma(v))\cdot
\gamma(v)^{-1}=z.$ Then $ \varphi(x_{0}) = (u, z, v).$ Hence,
$\varphi$ is a surjective map.

Therefore, $(\varphi, Id_{G_{0}})$ is an isomorphism of
group-groupoids.
\end{proof}

\vspace*{0.2cm}

\hspace*{0.7cm}West University of Timi\c soara\\
\hspace*{0.7cm} Department of Mathematics\\
\hspace*{0.7cm} Bd. V. P{\^a}rvan,no.4, 300223, Timi\c soara, Romania\\
\hspace*{0.7cm}E-mail: ivan@math.uvt.ro

\end{document}